\documentclass[reqno,a4paper,11pt]{amsart}
\usepackage{color}
\usepackage{amsmath,amssymb,amsfonts,amsthm,enumerate}
\usepackage{hyperref,mathrsfs,accents}
\usepackage[T1]{fontenc}
\usepackage[active]{srcltx}
\usepackage{epsfig}
\usepackage{graphicx}
\usepackage{cite}
\usepackage{tikz}
\usepackage{pgfplots}
\usepackage{tikz-3dplot}
\usepackage{subcaption}
\usepackage{pgfplots}
\usepackage{adjustbox}
\usetikzlibrary{patterns,backgrounds}

\tdplotsetmaincoords{60}{115}
\pgfplotsset{compat=newest}

\catcode`@=11 \@addtoreset{equation}{section} \catcode`@=12

\newtheorem{theorem}{Theorem}[section]
\newtheorem{lemma}[theorem]{Lemma}
\newtheorem{definition}[theorem]{Definition}
\newtheorem{remark}[theorem]{Remark}
\newtheorem{proposition}[theorem]{Proposition}

\newenvironment{Proof}{\removelastskip\par\medskip 
	\noindent{\em Proof.}
	\rm}{\penalty-20\null\hfill$\square$\par\medbreak}

\newcommand{\N}{\mathbb{N}}
\newcommand{\R}{\mathbb{R}}

\newcommand{\e}{\varepsilon}
\newcommand{\C}{{\mathcal C}}

\newcommand{\Hau}{{\mathcal H}} 
\newcommand{\de}{\partial}

\newcommand{\spt}{\mathop{\mathrm{spt}}}
\newcommand{\loc}{\mathop{\mathrm{loc}}}
\newcommand{\difsim}{\Delta}
\newcommand{\ch}{\mathbf{1}}

\newcommand{\cB}{\mathcal{B}}

\renewcommand{\Subset}{\subset\!\subset}

\newcommand{\pr}{P_{\Omega}}
\newcommand{\prz}{P_{\Omega_{0}}}
\newcommand{\prt}{P_{\Omega_{t}}}
\newcommand{\p}{P}
\newcommand{\dd}{d}
\newcommand{\tP}{\widetilde{p}}
\newcommand{\tQ}{\widetilde{q}}
\newcommand{\tN}{\widetilde{N}}
\newcommand{\tS}{\widetilde{S}}
\newcommand{\talpha}{\widetilde{\alpha}}
\newcommand{\tbeta}{\widetilde{\beta}}
\newcommand{\tgamma}{\widetilde{\gamma}}
\newcommand{\tSigma}{\widetilde{\Sigma}}
\newcommand{\Tr}{{\mathrm{Tr}}}

\AtBeginDocument{
	\let\div\relax
	\DeclareMathOperator{\div}{div}
}

\definecolor{grey}{rgb}{.7,.7,.7}
\definecolor{evidGP}{rgb}{0,0,1}
\definecolor{evidG}{rgb}{0,0.5,0}
\definecolor{almond}{rgb}{0.99, 0.99, 0.99}

\author{Gian Paolo Leonardi}
\address{Dipartimento di Matematica, via Sommarive 14, IT-38123 Povo - Trento (Italy)}
\email{gianpaolo.leonardi@unitn.it}

\author{Giacomo Vianello}
\address{Dipartimento di Matematica, via Sommarive 14, IT-38123 Povo - Trento (Italy)}
\email{giacomo.vianello-1@unitn.it}



\title[A vertex-skipping property...]{A vertex-skipping property for almost-minimizers of the relative perimeter in convex sets}

\thanks{G.P.Leonardi has been partially supported by: PRIN 2017TEXA3H ``Gradient flows, Optimal Transport and Metric Measure Structures''; PRIN 2022PJ9EFL ``Geometric Measure Theory: Structure of Singular Measures, Regularity Theory and Applications in the Calculus of Variations'' (financed by European Union - Next Generation EU, Mission 4, Component 2 - CUP:E53D23005860006); Grant PID2020-118180GB-I00 ``Geometric Variational Problems''. Giacomo Vianello has been supported by GNAMPA (INdAM) Project 2023: ``Esistenza e propriet\`a fini di forme ottime''. The authors wish to thank the anonymous referees for very helpful comments on the first version of the paper, and in particular for pointing out a remarkable consequence of our main result in combination with \cite[Theorem 1.1]{EdelenLi2022}.}

\begin{document}
\begin{abstract}
Given a convex domain $\Omega\subset \R^{3}$ and an almost-minimizer $E$ of the relative perimeter in $\Omega$, we prove that the closure of $\de E \cap \Omega$ does not contain vertices of $\Omega$. 
\end{abstract}

\maketitle
	
	
\section{Introduction}
Given a convex open set $\Omega \subset \R^{n}$, $n\ge 2$, we aim to study the boundary behavior of local almost-minimizers of the relative perimeter in $\Omega$ near certain singular points of $\de \Omega$. 
More specifically, we take a measurable set $E\subset \Omega$ and, for any open set $A\subset \R^{n}$, we define the relative perimeter of $E$ in $\Omega$ restricted to $A$ as 
\[
\pr(E;A) := \p(E;\Omega\cap A)\,,
\]
with the short form $\pr(E) := \p(E;\Omega)$ when $A\supset \Omega$, and with $P(E;B)$ denoting the standard perimeter of $E$ in $B$ à la De Giorgi. 
We then say that $E$ is a local almost-minimizer of $\pr$ if, for any $x\in \overline{\Omega}$ there exists $r_{x} > 0$ such that, for any $0 < r < r_{x}$ and any measurable subset $F$ of $\Omega$ with $F \difsim E \subset \subset B_{r}(x)$, one has
\begin{equation}\label{eq:A-M}
\pr(E;B_{r}(x)) \leq \pr(F;B_{r}(x)) + |F \difsim E|^{\frac{n-1}{n}} \psi_{\Omega}(E;x,r) \, ,
\end{equation}
for a suitable function $\psi_{\Omega}(E;x,r)$ such that $\lim_{r\to 0^{+}}\psi_{\Omega}(E;x,r) = 0$. In particular, $E$ is a perimeter minimizer in $B_{r}(x)$ if and only if $\psi_{\Omega}(E;x,r) = 0$. 

Similar notions of almost-minimality are well-known in the literature, and the most relevant ones satisfy condition \eqref{eq:A-M}. For instance, the lambda-minimality, which characterizes solutions to isoperimetric and prescribed mean curvature problems, implies \eqref{eq:A-M} by choosing $\psi_{\Omega}(E;x,r) = \Lambda r\omega_{n}^{1/n}$. It is worth recalling that De Giorgi's interior regularity for perimeter minimizers is well-known to hold up to $C^{1,\alpha}$-regularity for almost-minimizers in the sense of Tamanini \cite{tamanini1982boundaries, Tamanini_quadLecce}, which in our case corresponds to \eqref{eq:A-M} combined with the extra summability of $r^{-1}\psi(E;x,r)$ on some interval $(0,r_{0})$, for any $x\in \Omega$ and with $r_{0}>0$ possibly depending on $x$. 

Further regularity results are known in the context of integral currents and varifolds, where instead of almost-minimality hypotheses one considers summability assumptions on the generalized mean curvature of the (associated) varifold \cite{allard1975boundary,allard1972first}. In particular, the problem of boundary regularity for integral $k$-varifolds with free boundary and mean curvature in $L^{p}$ with $p>k$, has been first tackled in \cite{gruter1986allard}. 

To state our main result, we need to introduce the notion of vertex of $\Omega$. Given $x_{0}\in \de\Omega$ we define the (open) tangent cone to $\Omega$ at $x_{0}$ as
\[
T_{x_{0}}\Omega = \lim_{t\to +\infty}\ t (\Omega - x_{0}) = \bigcup_{t>0}t (\Omega - x_{0})\,.
\]
We say that $x_{0}$ is a vertex of $\Omega$ if $T_{x_{0}}\Omega$ does not contain lines. Notice that this is equivalent to require that, up to isometries, $T_{x_{0}}\Omega$ cannot be written as $\R\times C$ for some convex cone $C$ in $\R^{n-1}$. That said, our main result is the following
\begin{theorem}[Vertex-skipping]\label{thm:VS}
Let $\Omega \subset \R^3$ be a open, convex set. Let $E$ be a local almost-minimizer of the relative perimeter in $\Omega$. Then $\overline{\de E\cap \Omega}$ does not contain vertices of $\Omega$.
\end{theorem}
	
Theorem \ref{thm:VS} provides a partial answer to the broader question whether the internal boundary of an almost-minimizer of the relative perimeter in $\Omega$ can reach singular points of $\de\Omega$. The restriction to dimension $n=3$ is not purely technical. Indeed, in the forthcoming paper \cite{LeoVia_conoR4}, we show that in dimension $n=4$ the intersection of the half-space 
\[
E = \{(x_{1},\dots,x_{4})\in \R^{4}:\ x_{1}<0\}
\] 
with the cone 
\[
C_{\lambda} = \{(x_{1},\dots,x_{4})\in \R^{4}:\ x_{4}>\lambda\sqrt{x_{1}^{2}+x_{2}^{2}+x_{3}^{2}}\}
\] 
is strictly stable with respect to a suitable class of compact variations up to the boundary of $C_{\lambda}$, as soon as $\lambda>0$ is small enough. 
Even though the global minimality of $E$ for the relative perimeter in $C_{\lambda}$ for small $\lambda$ (with respect to \emph{all} compact variations) remains an open question, the strict stability of $E$ in $C_{\lambda}$ supports the conjecture that Theorem \ref{thm:VS} is false in dimension $n\ge 4$.

In the more general context of capillarity, it is well-known that when $\Omega$ is of class $C^{1,1}$, a capillary surface interface in $\Omega$ meets $\de\Omega$ according to Young's law \cite{young1805iii}, i.e., forming a contact angle $\beta = \arccos \gamma$. Here $\gamma\in [-1,1]$ is the wetting coefficient appearing in the capillary energy (without bulk terms)
\[
\pr(E) + \gamma \p(E;\de \Omega)\,,
\]
so that one obtains the relative perimeter in the special case $\gamma=0$.
The rigorous derivation of Young's law is based on a boundary regularity result for local minimizers of the capillary energy, which in turn requires $\de\Omega$ to be sufficiently smooth \cite{DePhilippisMaggi2015}. However, the study of capillarity in non-smooth containers is not only as relevant in applications, as it is Young's law for smooth ones (see \cite{Fin86book}), but also one of the first historical testimonies of scientific research around this natural phenomenon. Indeed, one of the first studies on the behavior of capillary surfaces in wedges was conducted by Taylor about a century before Young's and Laplace's works in the context of smooth containers (see \cite{Fin86book}). The interest for the capillarity phenomenon in containers with wedge- or corner-type singularities emerges in various works, like in Concus-Finn \cite{ConcusFinn69}, Lancaster \cite{Lancaster2012}, Chen-Finn-Miersemann \cite{ChenFinnMiersemann2008}, Lancaster-Siegel \cite{LancasterSiegel1996}, Tamanini \cite{Tamanini_angoli}. We also mention that, in the same spirit, the first author in collaboration with G.~Saracco studied in \cite{LS18a} the boundary behavior of non-parametric solutions of the prescribed mean curvature equation on a weakly-regular domain $A$, with application to the capillarity for perfectly wetting fluids ($\gamma=-1$) that partially fill the cylinder $A\times \R$. 
	
We now consider the case of the relative perimeter in convex domains, and recall some relevant results in this context. In \cite{BokowskiSperner1979}, Bokowski and Sperner obtained the characterization of minimizers for the relative isoperimetric problem in Euclidean balls, as well as a relative isoperimetric inequality in a convex domain with a non-sharp constant explicitly written in terms of a Bonnesen-type asymmetry index. Later on, Lions and Pacella \cite{LionsPacella1990} showed the optimal isoperimetric inequality 
\[
\p(E; \C) \ge \p(B_{E}; \C)
\] 
in a convex cone $\C$ with a vertex at the origin, for all measurable $E\subset \C$ with $|E|<+\infty$, where $B_{E}$ denotes the Euclidean ball centered at the origin and such that $|B_{E}\cap \C|=|E|$. 
The same inequality has been later strengthened by Figalli and Indrei \cite{FigalliIndrei2013} who proved its sharp quantitative form using optimal transport. We also mention the relative isoperimetric problem in $\Omega$ such that $\R^{n}\setminus \Omega$ is convex. This problem was first tackled by Choe, Ghomi and Ritor\'e \cite{ChoeGhomiRitore2007}, who proved the optimal inequality
\[
\pr(E) \ge \frac{n\omega_{n}^{\frac{1}{n}}}{2} |E\cap \Omega|^{(n-1)/n}
\]
plus the characterization of the equality case, under the assumption $\de\Omega \in C^{2}$. More recently, this result has been extended by Fusco and Morini \cite{FuscoMorini2023} to the non-smooth case. Concerning the properties of minimizers, Sternberg and Zumbrun proved in \cite{SternbergZumbrun1999} that the internal boundary of any local minimizer of the relative perimeter in a convex domain with $C^{2,\alpha}$ boundary is either connected, or a union of parallel planes meeting $\de\Omega$ orthogonally. Such a property follows quite directly from the concavity of the isoperimetric profile, which is another key result obtained in the very same paper. See also \cite{Kuwert2003}, \cite{Milman2009}, \cite{RitoreVernadakis2015}, \cite{LRV18} for extensions to general convex sets. 

The study of critical points of the relative perimeter is closely connected with the classification of free-boundary minimal surfaces. The internal boundary of a local minimizer of the relative perimeter without volume constraints is indeed a free-boundary minimal surface. The study of free-boundary minimal surfaces, initiated by seminal works of Courant \cite{Courant1945} and Lewy \cite{Lewy1951}, has been carried out by various authors in the past (see e.g. Nitsche's book \cite{Nitsche-book} for an updated list of classical references, up to 1975) and still represents a very active research area 
%
(see \cite{Carlotto2019} for a comprehensive overview of more recent advances in the field). However, the majority of classification results is usually obtained on specific domains - or ambient manifolds - with smooth boundary, while the case of non-smooth domains is typically neglected due to the technical difficulties in implementing variation arguments in presence of boundary singularities. It is, however, worth mentioning the very recent, Allard-type $\e$-regularity result for free-boundary minimal surfaces, recently proved by Edelen and Li in \cite{EdelenLi2022} for locally-almost-polyhedral domains. In this regard, Theorem \ref{thm:VS} can be easily extended to free-boundary, locally (almost) area-minimizing $2$-currents in $3$-dimensional convex domains. Moreover, as observed by one of the anonymous referees of the paper, if we combine our result with \cite[Theorem 1.1]{EdelenLi2022}, we can prove that the singular set of any such current in a polyhedral domain is empty, which improves \cite[Theorem 1.2]{EdelenLi2022}.

\section{Preliminaries}
\subsection{Notation.}
For this work, unless otherwise specified, $\Omega$ denotes a open, convex subset of $\R^{n}$. Given $A \subset \R^{n}$ open and $F \subset \Omega$ measurable, we define the relative perimeter of $F$ in $A$ as
\[
\pr(F;A) := \p(F;A \cap \Omega) \,,
\]
where 
\[
\p(F;A) = \sup \left\{\int_{A\cap F} \div g(x)\, dx\,:\ g\in C^{1}_{c}(A;\R^{n}),\ |g|\le 1\right\}\,.
\]
It is convenient to introduce the minimality gap of $E$ in $A$ (relative to $\Omega$) as
\[
\Psi_{\Omega}(E;A) = \pr(E;A) - \inf\big\{\pr(F;A):\ F\difsim E \Subset A\cap \overline{\Omega}\big\}\,.
\]
We observe that, if $E$ is a local almost-minimizer of $\pr$ then, for all $x \in \overline{\Omega}$ and for all $0 < r < r_x$,
	\begin{equation}\label{eq:psialmost}
		\Psi_{\Omega}(E;B_{r}(x)) \leq \omega_{n}^{\frac{n-1}{n}} r^{n-1} \psi_{\Omega}(E;x,r) \, .
	\end{equation}

We denote the characteristic function of $F$ by $\ch_{F}$. We denote by $Df$ the distributional gradient of $f \in BV(A)$, and identify it as usual with a vector-valued Radon measure on $A$. Given a vector-valued Radon measure $\mu = (\mu_{1},...,\mu_{p}): \cB(\R^{n}) \rightarrow \R^{p}$, we denote by $|\mu|$ its total variation. Let $u = (u_{1},...,u_{p}): \R^{n} \rightarrow \R^{p}$, $u \in L^{1}_{\loc}(\mu)$, then we denote by $u \cdot \mu$ the Radon measure defined by
\begin{equation*}
    (u \cdot \mu) E := \int_{E} u \cdot \dd \mu = \sum_{q = 1}^{p} \int_{E} u_{q} \dd \mu_{q} \, .
\end{equation*}
Let $S \subset \R^{n}$ and $t > 0$, we define $S_{t} := t^{-1} S$. Consequently, the tangent cone to a convex set $\Omega$ at $x_{0} \in \overline{\Omega}$ is denoted by $T_{x_{0}} \Omega$ and defined as the (set-theoretic, loc-Hausdorff, $L^{1}_{loc}$) limit of $(\Omega - x_{0})_{t}$ as $t\to 0$. We shall often assume that $x_{0}=0$ and abbreviate $\Omega_{0} = T_{0}\Omega$. 

Finally, we say that a convex set $C \subset \R^{n}$ is a wedge provided there exist $x_0\in \R^{n}$ and two linearly independent unit-normal vectors $w_{1}, w_{2} \in \de B_{1}$ such that
    \begin{equation*}
        C = \{ x \in \R^{n} : \, (x-x_0) \cdot w_{1} \leq 0 , \, (x-x_0) \cdot w_{2} \leq 0 \} \, .
    \end{equation*}
We note that a wedge $C$ is a cone with respect to all points of the $n-2$ dimensional affine subspace 
\[
\sigma_{C} = \{x_{0} + v\,:\ \langle v,w_{i}\rangle = 0\ \forall\, i=1,2\}\,,
\]
that we call the spine of $C$.

\subsection{Boundary density estimates.} Here we establish perimeter and volume density estimates for almost-minimizers at a boundary point for $\Omega$, and we provide the full proof for the reader's convenience. We highlight that, in what follows, $\Omega$ is required to be just Lipschitz (i.e., the convexity of $\Omega$ is not needed).
\begin{lemma}
Let $\Omega\subset \R^{n}$ be an open set with Lipschitz boundary. Let $E$ be an almost-minimizer in $\Omega$. Up to a translation, we assume that $0 \in \de\Omega$ and that $\pr(E;B_{r}) > 0$ for all $r > 0$. Then, there exist a constant $C \ge 1$ and a radius $\overline{r} > 0$, both depending on $0$, such that
\begin{align}
	\label{eq:densityper}
	&C^{-1}r^{n-1} \le \pr(E;B_{r}) \le Cr^{n-1}\\
	\label{eq:densityvol}
	&\min\big(|E\cap B_{r}\cap \Omega|,\ |(B_{r}\cap \Omega)\setminus E|\big) \ge C^{-1}r^{n}\,,
\end{align}
for all $0 < r < \overline{r}$.
\end{lemma}
\begin{Proof}
	We start proving \eqref{eq:densityvol}. Given $0<r<r_{0}$ we set
	\[
	m(r) := |B_{r} \cap \Omega\cap E| \, , \quad \mu(r) := |B_{r} \cap \Omega \setminus E| \,.
	\]
	Both $m$ and $\mu$ are non-decreasing, thus differentiable for almost all $r > 0$. By \cite[Example 13.3]{maggi2012sets}, for almost all $r > 0$, we have
	\[
	m'(r) = \Hau^{n-1}(E \cap \de B_{r} \cap \Omega) \, , \quad \mu'(r) = \Hau^{n-1}(\de B_{r} \cap \Omega \setminus E) \, .
	\]
Since the one-parameter family of rescaled domains $D_{r} := r^{-1}(\Omega\cap B_{r})$, $0<r\le 1$ is precompact in the class of Lipschitz and connected domains with respect to the $L^{1}$-convergence, there exists a constant $\overline{C}>0$ such that the following, relative isoperimetric inequality holds:
\begin{equation} \label{eq:relisop}
	\pr(E;B_{r}) \geq \overline{C} \min \{ m(r), \mu(r) \}^{\frac{n-1}{n}} \,,
\end{equation}
for all $0<r<1$. Set $0 < t < r$ and define the competitor 
\[
F_{t} =
\begin{cases}
E \cup B_{t} \cap \Omega & \text{if } m(t) > \mu(t),\\ 
E \setminus B_{t} \cap \Omega &  \text{otherwise.} 
\end{cases}
\]
We note that in the first case $F_{t} \difsim E = B_{t} \cap \Omega \setminus E$, while in the second case $F_{t} \difsim E = B_{t} \cap \Omega \cap E$. In any case, we have $F_{t} \difsim E \subset \subset B_{r} \cap \Omega$. Thus, by the almost-minimality of $E$ in $\Omega$, and for almost all $0<t<r$, we infer that either
	\begin{align} \label{2a}
		\pr(E;B_{r}) & \leq \pr(F_{t};B_{r}) + \mu(t)^{\frac{n-1}{n}} \psi(r) \\ 
		\nonumber & \leq \pr(E;B_{r} \setminus \overline{B_{t}}) +  \Hau^{n-1}(\de B_{t} \cap \Omega \setminus E) + \mu(r)^{\frac{n-1}{n}} \psi(r) \, ,
	\end{align}
	or
	\begin{align} \label{2b}
		\pr(E;B_{r}) & \leq \pr(F_{t};B_{r}) + m(t)^{\frac{n-1}{n}} \psi(r) \\ 
		\nonumber & \leq \pr(E;B_{r} \setminus \overline{B_{t}}) + \Hau^{n-1}(\de B_{t} \cap \Omega \cap E) + m(r)^{\frac{n-1}{n}} \psi(r) \, .
	\end{align}
	where $\psi(r) := \psi_{\Omega}(E;0,r)$. Taking the limit as $t \nearrow r$ in \eqref{2a} and \eqref{2b}, and using \eqref{eq:relisop}, we deduce that, if $m(r) > \mu(r)$, then for almost all $0<r<r_{0}$ we have
	\begin{equation*}
		\mu'(r) + \mu(r)^{\frac{n-1}{n}} \psi(r) \geq \overline{C} \mu(r)^{\frac{n-1}{n}} \, ,
	\end{equation*}
	while otherwise we have
	\begin{equation*}
	  m'(r) + m(r)^{\frac{n-1}{n}} \psi(r) \geq \overline{C} 
        m(r)^{\frac{n-1}{n}}\,.
	\end{equation*}
	Therefore, calling $s(r) := \min \{ m(r), \mu(r) \}$ and owing to the infinitesimality of $\psi(r)$ as $r\to 0$, we obtain
	\begin{equation*}
		\dfrac{s'(r)}{s(r)^{\frac{n-1}{n}}} \geq C \, ,
	\end{equation*}
	for $0 < r < \overline{r}$, for some $C,\overline{r} > 0$. Integrating this inequality on the interval $(0,r)$ we obtain \eqref{eq:densityvol}. Then, the first inequality in \eqref{eq:densityper} follows from \eqref{eq:densityvol} and \eqref{eq:relisop}. Finally, the second inequality in \eqref{eq:densityper} follows from the observation that, taking the limit as $t \nearrow r$ in \eqref{2a} and possibly redefining $\overline{r}$ and $C$, we have
	\begin{align*}
		\pr(E;B_{r}) & \leq \Hau(\de B_{r} \cap \Omega \setminus E) + \mu(r)^{\frac{n-1}{n}} \psi(r) \leq C r^{n-1} \, ,
	\end{align*}
	for every $0 < r < \overline{r}$.
\end{Proof}
	
\subsection{Blow-up limits.} We now show that a sequence of dilations of an almost-minimizer $E$ in $\Omega$ converge, up to subsequences, to a minimizer of the relative perimeter in the tangent cone $\Omega_{0}$. 
\begin{lemma} \label{BlowUpConv}
	Let $E$ be a almost-minimizer in $\Omega$, and assume $0 \in \de \Omega$ and $\pr(E;B_{r}) > 0$ for all $r > 0$. Then, there exist a sequence $t_{j} \searrow 0$ and a measurable set $E_{0} \subset \Omega_{0}$ such that $E_{t_{j}} \rightarrow E_{0}$ in $L^{1}_{loc}$ and $E_{0}$ is a perimeter-minimizer in $\Omega_{0}$, namely 
\[
\Psi_{\Omega_{0}}(E_{0};B_{R}) = 0 \qquad \text{for any $R > 0$.}
\]
\end{lemma}
	
\begin{proof}
First, we fix $R>0$ and prove that there exist $t_{0},C > 0$ such that
	\begin{equation} \label{eq:PerBound}
		\p(E_{t};B_{R}) \leq C R^{n-1}\qquad \forall\, 0<t<t_{0}\,.
	\end{equation}
In what follows, for more simplicity, we will write $C$ to denote a constant that might change from one line to another. To prove \eqref{eq:PerBound}, we note that  
\begin{equation}\label{eq:blowub}
\p(\Omega_{t};B_{R}) = t^{1-n}\p(\Omega;B_{Rt}) \le C R^{n-1}
\end{equation}
since $\de\Omega$ is Lipschitz. 
Then owing to \eqref{eq:densityper} and \eqref{eq:blowub}, and assuming $t< t_{0}:= \min(1,\bar{r}/R)$, we obtain
\begin{align*}
\p(E_{t};B_{R}) \leq \prt(E_{t};B_{R}) + \p(\Omega_{t};B_{R}) = t^{1-n} \big(\pr(E;B_{tR}) +  \p(\Omega;B_{Rt})\big) \le  C\, R^{n-1}\,,
\end{align*}
which proves \eqref{eq:PerBound}. Consequently, we obtain the global perimeter bound 
\[
\p(E_{t}) \le \p(E_{t};B_{R}) + \p(\Omega_{t};B_{R}) \le CR^{n-1}\,,
\] 
hence any blow-up sequence $E_{t_{j}}$ admits a (not relabeled) subsequence converging in $L^{1}(B_{R})$ to a limit set $E_{0}$. By a standard diagonal argument, one can prove the existence of a subsequence and a limit set, still denoted respectively as $E_{t_{j}}$ and $E_{0}$, such that $E_{t_{j}}\rightarrow E_{0}$ in $L^{1}_{loc}(\R^{n})$. Since the corresponding sequence of rescaled domains $\Omega_{t_{j}}$ converges to the tangent cone $\Omega_{0}$ in $L^{1}_{loc}(\R^{n})$, we infer that $E_{0}\subset \Omega_{0}$.  

Now we have to show that $\Psi_{\Omega_{0}}(E_{0};B_{R}) = 0$, for all $R > 0$. To do so, we claim that
\begin{equation} \label{LSCMinGap}
\Psi_{\Omega_{0}}(E_{0}; B_{R}) \leq \liminf_{j \rightarrow \infty} \Psi_{\Omega_{t_{j}}}(E_{t_{j}};B_{R}) \, .
\end{equation}
Indeed, let $F_{0} \subset \Omega_{0}$ be such that
\begin{equation*}
F_{0} \difsim E_{0} \subset \subset B_{R} \cap \overline{\Omega_{0}} \, .
\end{equation*}
By well-known properties of the trace of a $BV$ function, for all $j$, for a.e. $0 < \rho < R$ and $\Hau^{n-1}$-a.e. $x\in \de B_{\rho}$, we have 
\begin{equation}\label{eq:equaltraces}
\ch_{E_{t_{j}}}(x) = \Tr^{\pm}(E_{t_{j}},\de B_{\rho})(x)\quad \text{and} \quad  \ch_{E_{0}}(x) = \Tr^{\pm}(E_{0},\de B_{\rho})(x)\,,
\end{equation}
where $\Tr^{\pm}(E,\de B_{\rho})$ denotes the inner ($+$) or outer ($-$) trace of $\ch_{E}$ on $\de B_{\rho}$. By the $L^{1}_{loc}$-convergence of $E_{t_{j}}$ to $E_{0}$, we can choose $\rho<R$ with the above property, and such that 
\begin{equation} \label{CondRho}
		F_{0} \difsim E_{0} \subset \subset B_{\rho} \cap \overline{\Omega_{0}}\qquad\text{and}\qquad
		\lim_{j\to\infty}\int_{\de B_{\rho} \cap \Omega_{0}} |\ch_{E_{t_{j}}} - \ch_{E_{0}}| \dd \Hau^{n-1} = 0 \, .
\end{equation}
%
For any $j$ we define
	\begin{equation*}
		F_{j} := [(F_{0} \cap B_{\rho}) \cup \big(E_{t_{j}} \cap (B_{R} \setminus B_{\rho})\big)] \cap \Omega_{t_{j}} \, . 
	\end{equation*}
By construction, and thanks to \eqref{CondRho}, we have
\begin{equation*}
F_{j} \difsim E_{t_{j}} \subset \subset B_{R} \cap \overline{\Omega_{t_{j}}}
\end{equation*}
and
\begin{equation*}
\p(F_{j};B_{R} \cap \Omega_{t_{j}}) = \p(F_{0};B_{\rho} \cap \Omega_{t_{j}}) + \p(E_{t_{j}};(B_{R} \setminus \overline{B_{\rho}}) \cap \Omega_{t_{j}}) + \int_{\de B_{\rho} \cap \Omega_{t_{j}}} |\ch_{E_{t_{j}}} - \ch_{E_{0}}| \dd \Hau^{n-1} \, ,
\end{equation*}
hence owing to  \eqref{eq:equaltraces} we infer $\p(E_{t_{j}};\de B_{\rho}\cap \Omega_{t_{j}}) = 0$, and obtain
\begin{align} \label{EstMinGap}
		\Psi_{\Omega_{t_{j}}}(E_{t_{j}};B_{R}) & \geq \p(E_{t_{j}};B_{R} \cap \Omega_{t_{j}}) - \p(F_{j};B_{R} \cap \Omega_{t_{j}}) \nonumber \\
		& = \p(E_{t_{j}};B_{\rho} \cap \Omega_{t_{j}}) - \p(F_{0};B_{\rho} \cap \Omega_{t_{j}}) - \int_{\de B_{\rho} \cap \Omega_{t_{j}}} |\ch_{E_{t_{j}}} - \ch_{E_{0}}| \dd \Hau^{n-1} \nonumber \\
		& \geq \p(E_{t_{j}};B_{\rho} \cap \Omega_{t_{j}}) - \p(F_{0};B_{R} \cap \Omega_{0}) - \int_{\de B_{\rho} \cap \Omega_{t_{j}}} |\ch_{E_{t_{j}}} - \ch_{E_{0}}| \dd \Hau^{n-1} \, . 
	\end{align}
	Given $\e > 0$, we can select $j_{0}$ and possibly update the choice of $0<\rho <R$ in such a way that all the previous requirements are still satisfied, and moreover we have
	\begin{equation} \label{RadiusChoice}
		\p(E_{0};B_{\rho} \cap \Omega_{t_{j_{0}}}) \ge \p(E_{0};B_{R} \cap \Omega_{0}) - \e \, .
	\end{equation}
Now by \eqref{EstMinGap}, we infer that for $j \geq j_{0}$
	\begin{equation} \label{EstMinGap2}
		\Psi_{\Omega_{t_{j}}}(E_{t_{j}};B_{R}) \geq \p(E_{t_{j}};B_{\rho} \cap \Omega_{t_{j_{0}}}) - \p(F_{0};B_{R} \cap \Omega_{0}) - \int_{\de B_{\rho} \cap \Omega_{t_{j}}} |\ch_{E_{t_{j}}} - \ch_{E_{0}}| \dd \Hau^{n-1}
	\end{equation}
	Taking the $\liminf$ in \eqref{EstMinGap2}, and exploiting \eqref{CondRho}, \eqref{RadiusChoice} together with the lower-semicontinuity of the perimeter, we deduce that
	\begin{align*}
		\liminf_{j \rightarrow \infty} \Psi_{\Omega_{t_{j}}}(E_{t_{j}};B_{R}) & \geq \p(E_{0};B_{\rho} \cap \Omega_{t_{j_0}}) - \p(F_{0};B_{R} \cap \Omega_{0}) \\
		& \geq \p(E_{0};B_{R} \cap \Omega_{0}) - \p(F_{0};B_{R} \cap \Omega_{0}) - \e \, ,
	\end{align*}
which proves \eqref{LSCMinGap} by the arbitrary choice of $\e$ and $F_{0}$. 
Finally, by \eqref{eq:psialmost} we obtain
\begin{equation*}
		\Psi_{\Omega_{t_{j}}}(E_{t_{j}};B_{R}) = \dfrac{1}{t_{j}^{n-1}} \Psi_{\Omega}(E;B_{t_{j} R}) \leq \omega_{n}^{1 - \frac{1}{n}} \, R^{n-1} \psi_{\Omega}(E; 0, t_{j} R) \longrightarrow 0 \,,
\end{equation*}
which concludes the proof.
\end{proof}

\subsection{Boundary monotonicity formula in Lipschitz cones.}
We shall need the following conical deviation estimate for a set $G$ minimizing the relative perimeter in a Lipschitz cone $\Omega_0$.
\begin{theorem} \label{thm:monot}
	Let $G$ be a finite perimeter set in $\Omega_{0}$, and let $\cB_{r} := \Omega_{0} \cap B_{r}$. Then 
    \begin{align} \label{eq:MonotSet}
		& \left\{ \left| \dfrac{x}{|x|^{n}} \cdot D \ch_{G} \right| (\cB_{r} \setminus \cB_{\rho})  \right\}^{2}\ \leq\ 2 \,  |x|^{1-n} |D \ch_{G}|(\cB_{r} \setminus \cB_{\rho}) \\ 
		& \qquad \left\{ r^{1-n} \, |D \ch_{G}|(\cB_{r}) - \rho^{1-n} \, |D \ch_{G}|(\cB_{\rho}) + (n-1) \, \int_{\rho}^{r} t^{-n} \, \Psi_{\Omega_{0}}(G;B_{t}) \dd t \right\} \, , \nonumber
	\end{align}
    for almost every $0 < \rho < r$. 
\end{theorem}
Theorem \ref{thm:monot} is well-known under the assumption that the origin is an internal point, and $B_{r}\subset \Omega$, see \cite{Giu84book}, \cite{maggi2012sets}. When the origin is a vertex of $\Omega_{0}$, the classical proof is adapted without much effort, and we provide it here for the reader's convenience.
\begin{remark}\label{rem:monot}
We note that, if $G$ is a perimeter minimizer in $\Omega_{0}$, then $\Psi_{\Omega_{0}}(E;B_{r}) = 0$, hence by \eqref{eq:MonotSet} one obtains that $r^{1-n}\prz(G,B_{r})$ is non-decreasing w.r.t. $r$.
\end{remark}

\begin{proof}[Proof of Theorem \ref{thm:monot}]
We follow the proof of \cite[Lemma 5.8]{Giu84book}. First, we take $f \in \C^{\infty}(\cB_{R}) \cap BV(\cB_{R})$ and prove that, for almost all $0 < \rho < r < R$, one has
	\begin{align} \label{MonotBV}
		& \left\{ \left| \dfrac{x}{|x|^{n}} \cdot Df \right| (\cB_{r} \setminus \cB_{\rho})  \right\}^{2} \leq 2 \,  |x|^{1-n} \cdot |Df|(\cB_{r} \setminus \cB_{\rho}) \cdot \\ 
		& \left\{ r^{1-n} \, |Df|(\cB_{r}) - \rho^{1-n} \, |Df|(\cB_{\rho}) + (n-1) \, \int_{\rho}^{r} t^{-n} \, \Psi_{\Omega_{0}}(f;t) \dd t \right\} \, , \nonumber
	\end{align}
where 
\begin{equation}\label{eq:monforfunct}
\Psi_{\Omega_{0}}(f;t) \equiv |Df|(\cB_{t}) - \inf \{ |Dg|(\cB_{t}) :\ g \in BV(\cB_{t}), \, \spt(g - f) \subset \subset B_{t} \cap \overline{\Omega_{0}} \} \, .
\end{equation}
For $0 < t < R$, let
	\begin{equation*}
		f_{t}(x) = 
		\begin{cases}
			f(x) &\text{if $t < |x| < R$} \\
			f \left( t \dfrac{x}{|x|} \right) &\text{if $0 < |x| < t$.}
		\end{cases}
	\end{equation*}
Standard computations yield
	\begin{equation*}
		\int_{\cB_{t}} |Df_{t}| \dd x = \dfrac{t}{n-1} \, \int_{\de B_{t} \cap \Omega_{0}} |Df| \left\{ 1 - \dfrac{\langle x, Df \rangle^{2}}{|x|^{2} |Df|^{2}} \right\}^{1/2} \dd \Hau^{n-1} \, .
	\end{equation*}
	By definition of $\Psi_{\Omega_{0}}(f;t)$, we have
	\begin{align*}
		\int_{\cB_{t}} |Df| \dd x - \Psi_{\Omega_{0}}(f;t) & \leq \int_{\cB_{t}} |Df_{t}| \dd x \\
		& = \dfrac{t}{n-1} \, \int_{\de B_{t} \cap \Omega_{0}} |Df| \left\{ 1 - \dfrac{\langle x, Df \rangle^{2}}{|x|^{2} |Df|^{2}} \right\}^{1/2} \dd \Hau^{n-1} \\
		& \leq \dfrac{t}{n-1} \, \int_{\de B_{t} \cap \Omega_{0}} |Df| \dd \Hau^{n-1} - \dfrac{t}{2(n-1)} \, \int_{\de B_{t} \cap \Omega_{0}} \dfrac{\langle x, Df \rangle^{2}}{|x|^{2} |Df|} \dd \Hau^{n-1} \, .
	\end{align*}
	Multiplying both sides by $(n-1) \, t^{-n}$, we get
	\begin{align} \label{DevConEst}
		\dfrac{t^{1-n}}{2} \, \int_{\de B_{t} \cap \Omega_{0}} \dfrac{\langle x, Df \rangle^{2}}{|x|^{2} |Df|} \dd \Hau^{n-1} & \leq t^{1-n} \, \int_{\de B_{t} \cap \Omega_{0}} |Df| \dd \Hau^{n-1} - \\
		& - (n-1) \, t^{-n} \, \int_{\cB_{t}} |Df| \dd x + (n-1)t^{-n} \, \Psi_{\Omega_{0}}(f;t) \nonumber \\
		& = \dfrac{\dd}{\dd t} \left( t^{1-n} \, \int_{\cB_{t}} |Df| \dd x \right) + (n-1)t^{-n} \, \Psi_{\Omega_{0}}(f;t) \nonumber \, ,
	\end{align}
and, by integrating \eqref{DevConEst} with respect to $t$ between $\rho$ and $r$, we obtain
	\begin{equation} \label{Int1}
		\dfrac{1}{2} \, \int_{\cB_{r} \setminus \cB_{\rho}} \dfrac{\langle x, Df \rangle^{2}}{|x|^{n+1} |Df|} \dd x \leq r^{1-n} \, \int_{\cB_{r}} |Df| \dd x - \rho^{1-n} \, \int_{\cB_{\rho}} |Df| \dd x + (n-1) \, \int_{\rho}^{r} t^{-n} \, \Psi_{\Omega_{0}}(f;t) \dd t \, .
	\end{equation}
By H\"older's inequality, we find that
	\begin{equation} \label{Int2}
		\left\{ \int_{\cB_{r} \setminus \cB_{\rho}} \dfrac{|\langle x, Df \rangle|}{|x|^{n}} \dd x \right\}^{2} \leq \left\{ \int_{\cB_{r} \setminus \cB_{\rho}} |x|^{1-n} |Df| \dd x \right\} \left\{ \int_{\cB_{r} \setminus \cB_{\rho}} \dfrac{\langle x, Df \rangle^{2}}{|x|^{n+1} |Df|} \dd x \right\} \, .
	\end{equation}
Then, the combination of \eqref{Int1} and \eqref{Int2} gives the desired estimate \eqref{MonotBV}. Finally, we can approximate $\chi_{E}$ by a sequence of smooth $BV$ functions and pass to the limit in \eqref{MonotBV}, thus getting \eqref{eq:MonotSet}.
\end{proof}
\section{Flatness of the blow-up of an almost-minimizer}
Thanks to Lemma \ref{BlowUpConv}, we know that any blow-up sequence of the almost-minimizer $E$ in $\Omega$, with respect to $0 \in \de \Omega$, converges to a minimizer $E_{0}$ in the tangent cone $\Omega_{0}$, up to subsequences. Applying Lemma \ref{BlowUpConv} to $E_{0}$, we deduce the existence of a sequence $s_{j} \searrow 0$ and a set $E_{00}$ minimizing the relative perimeter in $\Omega_{0}$, such that
\begin{equation}\label{eq:E00}
(E_{0})_{s_{j}} \longrightarrow E_{00} \qquad \text{in $L^{1}_{loc}(\Omega_{0})$.}
\end{equation}
Actually, up to a diagonal argument, one can even show that there exists a suitable sequence $t_{j}\searrow 0$ such that $E_{t_{j}} \to E_{00}$ in $L^{1}_{loc}(\R^{n})$. A consequence of Theorem \ref{thm:monot} is the following
\begin{proposition} \label{thm:mincone}
The set $E_{00}$ is a perimeter-minimizing cone with a vertex at the origin. 
\end{proposition}
\begin{proof}
The argument requires a slight variant of the proof of \cite[Theorem 9.3]{Giu84book}. Let $s_{j} \searrow 0$ be such that $(E_{0})_{s_{j}}$ converges to $E_{00}$ in $L^{1}_{loc}(\R^{n})$. For $t > 0$, we let $p(t) := t^{1-n} \, \prz(E_{0},B_{t})$ and note that, by Remark \ref{rem:monot}, $p$ is non-decreasing in $t$ and, for all $R > 0$, we have
	\begin{equation*}
		p(tR) = R^{1-n}\prz((E_{0})_{t},B_{R}) \, .
	\end{equation*}
Consequently, as the perimeter measure of $(E_{0})_{s_{j}}$ weakly-$*$ converges to that of $E_{00}$ by a well-known property of perimeter almost--minimizers (see, e.g., \cite{tamanini1982boundaries}), we deduce that, for almost every $R>0$, 
    \begin{equation}\label{eq:indepratio}
        R^{1-n} \prz(E_{00},B_R) = \lim_{j\to\infty} p(s_j R)\,.      
    \end{equation}
    At the same time, the monotonicity of $p(t)$ implies that
    \[
    \lim_{j\to\infty} p(s_j R) = \lim_{t\to 0^+}p(t)\,,
    \]
    hence the limit does not depend on $R$. Then by approximating a generic radius $R$ by means of sequences of smaller/larger radii for which \eqref{eq:indepratio} is satisfied, one easily obtains that 
\[
R^{1-n} \prz(E_{00},B_R) = \prz(E_{00},B_1)\qquad \text{for all $R>0$.}
\]
We can now apply \eqref{eq:MonotSet} and get
\[
\left| x \cdot D\ch_{E_{00}} \right| (B_{r} \setminus B_{\rho}) = 0 \qquad \text{for all $0 < \rho < r$.}
\]
This implies that
\[
\langle x, \nu_{E_{00}}(x) \rangle = 0\qquad \text{for $\Hau^{n-1}$-a.e. $x \in \de^{*} E_{00}$.}
\]
By \cite[Proposition 28.8]{maggi2012sets}, we conclude that $E_{00}$ is, up to null sets, a cone with vertex at the origin.
\end{proof}

\subsection{Characterization of the conical minimizer in $\R^{3}$.} Starting from an almost-minimizer $E$ in $\Omega$ which satisfies volume density estimates at the origin, and applying Lemma \ref{BlowUpConv} at most twice, we have obtained a conical minimizer $E_{00}$ of the relative perimeter in the tangent cone $\Omega_{0}$. The next theorem shows that, in dimension $n=3$, $\de E_{00} \cap \Omega_{0}$ coincides with a convex angle contained in a $2$-plane through the origin, that meets $\de \Omega_{0}$ orthogonally. 
\begin{theorem} \label{thm:geod}
Let $n=3$ and $E_{00}$ be the conical minimizer obtained in the previous subsection. Then $\de E_{00}\cap \Omega_{0}$ coincides with a $2$-plane intersected with $\Omega_{0}$, that meets $\de \Omega_{0}\setminus \{0\}$ orthogonally. 
\end{theorem}
\begin{proof}
Set $F=E_{00}$ for brevity, then the proof is accomplished by showing that there exists exactly one geodesic arc $\gamma \subset \de B_{1} \cap \Omega_{0}$ such that
\begin{equation} \label{eq:geodesic}
\de F \cap \de B_{1} \cap \Omega_{0} = \gamma \, ,
\end{equation}
and $\gamma$ meets $\de \Omega_{0}\cap \de B_{1}$ orthogonally. 
We split the proof into some steps.

\textit{Step 1.} We claim that $\de F \cap \de B_{1} \cap \Omega_{0}$ is made of countably-many (open) geodesic arcs $\gamma_{i}$, $i\in \N$, such that
\begin{equation*}
\gamma_{i} \cap \gamma_{j} = \emptyset, \, \text{for $i \neq j$,} \quad \bigcup_{i} \gamma_{i} = \de F \cap \de B_{1} \cap \Omega_{0} \, .
\end{equation*}
By interior regularity, $\de F \cap \Omega_{0}$ is smooth, and its outer normal vector $\nu_{F}$ is orthogonal to the radial directions (recall that $F$ is a cone with vertex at the origin). Hence, $\de F$ intersects transversally $\de B_{1} \cap \Omega_{0}$ along smooth curves $\gamma_{i}$ that cannot cross each other. Since $F$ has locally-finite perimeter, the family of these curves is at most countable. Let us now show that $\gamma_{i}$ is a geodesic arc, for all $i$. With a slight abuse of notation, we identify $\gamma_i$ with its arc-legth parametrization defined on the interval $(0,L_{i})$, where $L_i$ is the length of the curve. The connected component of $\de F$ that intersects $\de B_1$ along $\gamma_{i}$ can be then parametrized through
\begin{equation} \label{eq:pargeo}
\sigma_{i}(s,t) := s \, \gamma_{i}(t) \, \text{, for $s>0$, $t \in (0,L_{i})$.}
\end{equation}
We can choose the parametrization $\gamma_{i}(t)$ in such a way that
$\nu_{F}(\sigma_{i}(s,t)) = \gamma_i(t) \times \gamma_i'(t)$, for all $s>0$. Exploiting \eqref{eq:pargeo}, and using $\div_{\de F} \nu_{F} = 0$ by the minimality of $F$, we infer
\begin{align} \label{eq:noproj}
0 &= \div_{\de F} \nu_{F} (\sigma_{i}(s,t)) \\
& = \dfrac{d}{dt}(\gamma_i(t) \times \gamma_i'(t)) \cdot \gamma_i'(t) \nonumber \\
& = - \gamma_i(t) \times \gamma_i'(t) \cdot \gamma_i''(t),\qquad \text{for all }t\,.
\end{align}
Since we also have $\gamma_i'(t) \cdot \gamma_i''(t) = 0$ by the choice of the arc-length parametrization, and observing that $\{\gamma_i'(t),\gamma_i(t) \times \gamma_i'(t)\}$ is an orthonormal basis for the tangent space to $\de B_{1}$ at $\gamma(t)$, we conclude that $\gamma_i''(t)$ is orthogonal to the tangent space to $\de B_{1}$ at $\gamma(t)$, which is precisely the definition of geodesic arc.
%
    \medskip 

\textit{Step 2.} We prove that  $\gamma_{i}$ meets $\de\Omega_{0}$ orthogonally at its endpoints. More precisely, if $p$ is an endpoint of $\gamma_i$, then $\Omega_0$ admits a unique supporting plane at $p$, hence the outer unit normal vector $\nu_{0}(p)$ to $\de \Omega_0$ at $p$ is well-defined, and moreover if we denote by $\nu_{i}$ the constant unit outer normal to the connected component of $\de F$ containing $\gamma_i$, we have
    \begin{equation} \label{eq:Young}
        \nu_{i} \cdot \nu_{0}(p) = 0 \, .
    \end{equation}
    Let us first prove that $\Omega_{0}$ admits a unique supporting plane at $p$. Up to a rotation, we may assume that $p = (0,0,1)$, hence it follows that $\nu_{i} \cdot e_{3} = 0$. Owing to Lemma \ref{thm:FedererReduction} below, we can find a sequence $t_{j} \searrow 0$ such that $\Omega_{0}^{p,t_{j}} := t_{j}^{-1}(\Omega_{0} - p)$ locally converge to a cylinder of type $C \times \R$, and $F^{p,t_{j}} := t_{j}^{-1}(F - p)$ locally converge to a cylinder that can be written as $G \times \R$, where $G \subset C$. Moreover, both $C$, and $G$ are cones with respect to $0$ in the plane $z = 0$, and $G$ is perimeter-minimizing in $C$. By convexity of $C$, up to a further rotation, we can assume that
    \begin{equation*}
        C = \{ (x_{1},x_{2},0) : \, x_{2} > \lambda |x_{1}| \} \, ,
    \end{equation*}
    for some $\lambda \geq 0$. Clearly, $\lambda = 0$ if and only if $p$ admits a unique supporting plane for $\Omega_{0}$. The only possibility is then that $\partial G\cap C$ is made of finitely many half-lines $L_{1},...,L_{k}$ of the form $L_{j} = \{ t \, v_{j} : \, t \geq 0 \}$, for some unit vectors $v_{j}$. Up to relabeling, we can assume 
    \begin{equation*}
        0< v_{1} \cdot e_{2} \leq v_{j} \cdot e_{2} \quad \text{for all }j\,.
    \end{equation*}
If either $\lambda > 0$ or $k>1$, we could replace an initial portion of $L_{1}$ with a projection segment onto the closest side of $C$, which strictly decreases the perimeter and thus contradicts the fact that $G$ is perimeter-minimizing in $C$. Hence we necessarily have $\lambda = 0$ and $k=1$, i.e., there exists a unique supporting plane to $\Omega_{0}$ at $p$ with $\nu_{0}(p)=-e_{2}$, and moreover $v_{1}=e_{2}$. This proves the claim and, additionally, shows that two different geodesic arcs cannot share a common endpoint.
\medskip

\textit{Step 3.} Finally, we prove that $\de F \cap \de B_{1} \cap \Omega_{0}$ is made of exactly one geodesic arc. 

Suppose by contradiction that there exist two geodesic arcs $\gamma_{1}\neq \gamma_{2}$ contained in $\de F \cap \de B_{1} \cap \Omega_{0}$. By the previous steps we know that 
\begin{equation}\label{eq:g1g2empty}
\overline{\gamma_{1}}\cap \overline{\gamma_{2}} = \emptyset\,.
\end{equation}
For $i=1,2$ we denote by $\Pi_{i}$ the plane through the origin that contains $\gamma_{i}$, and by $p_{i}$, $q_{i}$ the boundary points of $\gamma_{i}$  (see Figure \ref{tab:geod}). 

Then, we consider the point $N_{i}\in \de B_{1}$ such that $\gamma_{i}$ is contained in the equator with north pole $N_{i}$ and south pole $S_{i}=-N_{i}$, and denote by $\mu_{p_{i}}, \mu_{q_{i}}$ the corresponding meridians connecting $N_{i}$ with $S_{i}$ and passing through $p_{i}$ and $q_{i}$, respectively. These meridians bound a region of $\de B_{1}$, that we denote as $\Sigma_{i}$, which satisfies 
\[
\Omega_{0}\cap \de B_{1} \subset \Sigma_{i},\qquad i=1,2\,.
\]

\tdplotsetmaincoords{70}{30}
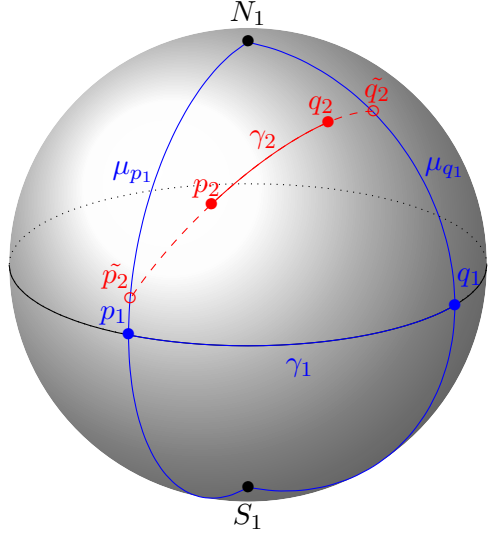
\begin{figure}
\begin{tikzpicture}[tdplot_main_coords,declare function={R=pi;}]
	\shade[tdplot_screen_coords,ball color=almond,opacity=0.5] (0,0) coordinate(O)
	circle[radius=R];
	\draw plot[variable=\x,domain=\tdplotmainphi-180:\tdplotmainphi,smooth]
	({R*cos(\x)},{R*sin(\x)},0);
	\draw[blue,pattern color=blue] 
	plot[variable=\x,domain=180:00,smooth] (0,{-R*sin(\x)},{R*cos(\x)})
	coordinate   (p1)
    plot[variable=\x,domain=0:180,smooth] ({R*sin(\x)},0,{R*cos(\x)})
	coordinate   (p2)
	plot[variable=\x,domain=0:90,smooth] ({R*cos(\x)},{-R*sin(\x)},0)
	coordinate   (p3);
	\draw[black,dotted,pattern color=black]
	plot[variable=\x,domain=-270:00,smooth] ({R*cos(\x)},{-R*sin(\x)},0)
	coordinate;
	\draw[red,dashed,color=red]
	plot[variable=\x,domain=-8:96,smooth] ({R*cos(\x)*(0.4472135955)},{-R*sin(\x)*(0.894427191) + R*cos(\x)*(0.2)*(0.8)},{R*cos(\x)*(0.8)*(0.894427191) + R*sin(\x)*(0.2)})
	coordinate;
	\draw[red,pattern color=red]
	plot[variable=\x,domain=15:60,smooth] ({R*cos(\x)*(0.4472135955)},{-R*sin(\x)*(0.894427191) + R*cos(\x)*(0.2)*(0.8)},{R*cos(\x)*(0.8)*(0.894427191) + R*sin(\x)*(0.2)})
	coordinate;

	\begin{scope}[on background layer]
		\foreach \X in {1,2}
		{ \draw[dotted] (O) -- (p\X); }
	\end{scope}
    \filldraw[black] (0,0,{pi+0.1}) node[anchor=south]{$N_{1}$};
    \filldraw[black] (0,0,{-pi-0.1}) node[anchor=north]{$S_{1}$};
    \filldraw[black] (0,-2.9,0.2) node[anchor=east]{\textcolor{blue}{$p_{1}$}};
    \filldraw[black] (3.72,0,0.5) node[anchor=east]{\textcolor{blue}{$q_{1}$}};
    \filldraw[black] (3,-3.2,0.1) node[anchor=east]{\textcolor{blue}{$\gamma_{1}$}};
    \filldraw[black] (0,-2.2,2) node[anchor=east]{\textcolor{blue}{$\mu_{p_{1}}$}};
    \filldraw[black] (3.45,0,2) node[anchor=east]{\textcolor{blue}{$\mu_{q_{1}}$}};
    \filldraw[black] (1.3,-1.2121,2.38) node[anchor=east]{\textcolor{red}{$\gamma_{2}$}};
    \filldraw[black] (1,-2.1810,1.95) node[anchor=east]{\textcolor{red}{$p_{2}$}};
    \filldraw[black] (1.6,-0.241612,2.62) node[anchor=east]{\textcolor{red}{$q_{2}$}};
    \filldraw[black] (0,-2.84563,0.73) node[anchor=east]{\textcolor{red}{$\tilde{p_{2}}$}};
    \filldraw[black] (1.8,0.88837,2.515) node[anchor=east]{\textcolor{red}{$\tilde{q_{2}}$}};
    \foreach \Point in {(0,0,pi),(0,0,-pi)}{
    	\node at \Point {\textbullet};
    };
    \foreach \Point in {(0,-pi,0),(pi,0,0)}{
    	\node at \Point {\textcolor{blue}{\textbullet}};
    };
    \foreach \Point in {(1.3564,-0.241612,2.332781),(0.7021,-2.1810,1.6672)}{
    	\node at \Point {\textcolor{red}{\textbullet}};
    };
    \foreach \Point in {(1.3905,0.88837,2.137534),(-0.14678,-2.84563,0.389705)}{
    	\node at \Point {\textcolor{red}{$\circ$}};
    };
\end{tikzpicture}
\caption{A possible geometric configuration in the contradiction argument for the proof of Step 3.}\label{tab:geod}
\end{figure}

Incidentally, thanks to the previous step, $\Sigma_{i}$ is also obtained by intersecting the sphere $\de B_{1}$ with the wedge $W_{i}$ given by the intersection of the two supporting half-spaces to $\Omega_{0}$ at $p_{i}$ and $q_{i}$, respectively. Since in particular $W_{i}$ is a convex cone, it is immediate to check that $\Sigma_{i} = W_{i}\cap \de B_{1}$ is geodesically convex. 
Moreover, using the fact that the angle formed by the vectors $p_{i},q_{i}$ is strictly smaller than $\pi$ (recall that the origin is an isolated vertex for $\Omega_{0}$) we infer that the internal angle formed by the two geodesic sides of $\Sigma_{i}$, i.e. the meridians $\mu_{p_{i}}$ and $\mu_{q_{i}}$, at $N_{i}$ (or $S_{i}$) is strictly smaller than $\pi$. 

Now, observe that the closure of $\Sigma_{i}$ is the union of two closed geodesic triangles $T_{N_{i}},T_{S_{i}}$ with vertices $p_{i},q_{i},N_{i}$ and $p_{i},q_{i},S_{i}$, respectively. 
Since in particular $\gamma_{2}\subset \Omega_{0}\cap \de B_{1}\subset \Sigma_{1}$, and $\gamma_{2}$ has a strictly positive distance from $\gamma_{1}$, we must have that either $\gamma_{2}\subset T_{N_{1}}$ or $\gamma_{2}\subset T_{S_{1}}$. Without loss of generality, we assume $\gamma_{2}\subset T_{N_{1}}$. 

Now, we set $\tP_{2} = \Pi_{2}\cap \mu_{p_{1}}$ and $\tQ_{2} = \Pi_{2}\cap \mu_{q_{1}}$, and denote by $\tgamma_{2}$ the geodesic connecting $\tP_{2}$ and $\tQ_{2}$, and by $\tSigma_{2}$ the associated geodesically convex region bounded by the meridians $\mu_{\tP_{2}},\mu_{\tQ_{2}}$ meeting at poles $\tN_{2} = N_{2}$ and $\tS_{2}=S_{2}$. Clearly we have $\gamma_{2}\subset \tgamma_{2}$ and thus $\Sigma_{2}\subset \tSigma_{2}$. Moreover, we have
\[
\tP_{2},\tQ_{2}\in T_{N_{1}}\setminus \{p_{1},q_{1},N_{1}\}\,.
\]
Indeed, the geodesic $\tgamma_{2}$ cannot intersect $\gamma_{1}$, hence it is contained in $T_{N_{1}}$ and its closure is disjoint from $\gamma_{1}$ because $\Pi_{2}\cap \overline{\gamma_{1}} = \overline{\gamma_{2}}\cap \overline{\gamma_{1}}=\emptyset$, by \eqref{eq:g1g2empty}; moreover, if we had $\tP_{2} = N_{1}$ (or $\tQ_{2}=N_{1}$) we would conclude that $\mu_{q_{1}}\subset \Pi_{2}$ (respectively, $\mu_{p_{1}}\subset \Pi_{2}$), but this is impossible because $\Pi_{2}$ and $\overline{\gamma_{1}}$ are disjoint.

Now, consider the geodesic quadrilateral $D$ determined by the four points $p_{1},\tP_{2},\tQ_{2},q_{1}$. By the previous argument, $D\subset T_{N_{1}}$. Denote by $\alpha_{1},\talpha_{2},\tbeta_{2},\beta_{1}$ the angles formed by the pairs of geodesic sides meeting at the respective vertices. Then, consider the two geodesic triangles $R_{1} = p_{1}\tP_{2}q_{1}$ and $R_{2} = \tP_{2}\tQ_{2}q_{1}$. Call $\talpha_{2,1}$ the internal angle to $R_{1}$ at $\tP_{2}$, and $\talpha_{2,2}$ the internal angle to $R_{2}$ at $\tP_{2}$. Similarly, call $\beta_{1,1}$ the internal angle to $R_{1}$ at $q_{1}$, and $\beta_{1,2}$ the internal angle to $R_{2}$ at $q_{1}$. We thus have $\talpha_{2} = \talpha_{2,1}+\talpha_{2,2}$, $\beta_{1} = \beta_{1,1}+\beta_{1,2}$, and therefore we deduce
\begin{align}
\label{eq:abpimezzi}
\alpha_{1} = \beta_{1} = \pi/2\,,\\
\label{eq:tatbmenopimezzi}
\max(\talpha_{2},\tbeta_{2}) \le \pi/2\,.
\end{align}
Indeed, \eqref{eq:abpimezzi} follows from the orthogonality of $\gamma_{1}$ with the meridians $\mu_{p_{1}}$ and $\mu_{q_{1}}$, while \eqref{eq:tatbmenopimezzi} follows from the fact that the quadrilateral $D$ is contained in one of the two geodesic triangles $\widetilde{T}_{N_{2}} = \tP_{2}\tQ_{2}N_{2}$, $\widetilde{T}_{S_{2}} = \tP_{2}\tQ_{2}S_{2}$ (indeed, by a symmetric argument, we have either $\gamma_{1}\subset \widetilde{T}_{N_{2}}$ or $\gamma_{1}\subset \widetilde{T}_{S_{2}}$) and we know by construction that the internal angles to $\widetilde{T}_{N_{2}}$ (or $\widetilde{T}_{S_{2}}$) at $\tP_{2}$ and at $\tQ_{2}$ are both equal to $\pi/2$.

Now we notice that $R_{1}$ must be a non-degenerate geodesic triangle, because it possesses an internal angle at $p_{1}$ measuring $\alpha_{1} = \pi/2$, and the other two vertices do not coincide. Thus we have that the sum of the internal angles of $R_{1}$ satisfies
\begin{equation}\label{eq:angoliR1}
\alpha_{1} + \talpha_{2,1} + \beta_{1,1} > \pi\,.
\end{equation} 
At the same time, the sum of the internal angles of $R_{2}$ is not smaller than $\pi$:
\begin{equation}\label{eq:angoliR2}
\talpha_{2,2} + \tbeta_{2} + \beta_{1,2} \ge \pi\,.
\end{equation} 
By combining \eqref{eq:abpimezzi}, \eqref{eq:tatbmenopimezzi}, \eqref{eq:angoliR1} and \eqref{eq:angoliR2}, we reach the contradiction
\[
2\pi < (\alpha_{1} + \talpha_{2,1} + \beta_{1,1}) + (\talpha_{2,2} + \tbeta_{2} + \beta_{1,2}) = \alpha_{1} + \talpha_{2} + \tbeta_{2} + \beta_{1} \le 2\pi
\]
and this completes the proof of the theorem.
\end{proof}

Next we state a slight variant of the classical Federer's Reduction Lemma, which is used in the proof of Theorem \ref{thm:geod}. In what follows, by ``cone'' we shall always mean a cone with respect to the origin. 
\begin{lemma} \label{thm:FedererReduction}
Let $K \subset \R^{n}$ be a convex cone, $C \subset K$ be a minimizing cone for the relative perimeter. Let $x_{0} \in \de C \cap \de K \setminus \{ 0 \}$. For $t > 0$, we set
    \begin{equation*}
    C_{t} := x_{0} + \dfrac{C - x_{0}}{t} \, , \quad K_{t} := x_{0} + \dfrac{K - x_{0}}{t} \, .
    \end{equation*}
    Then there exist a sequence $t_{j} \searrow 0$ and two sets $C_{0}$, $K_{0}$ such that
    \begin{equation} \label{eq:convergence}
        C_{j} := C_{t_{j}} \longrightarrow C_{0} \, , \quad K_{j} := K_{t_{j}} \longrightarrow K_{0}
    \end{equation}
    in $L^{1}_{\loc}$-topology, $C_{0} \subset K_{0}$ and $C_{0}$, $K_{0}$ are cylinders with axis coinciding with the line joining $0$ to $x_{0}$. Moreover, the sets $C_{0}' := C_{0} \cap x_{0}^{\perp}$ and  $K_{0}' := K_{0} \cap x_{0}^{\perp}$ are cones with respect to $0$ in the hyperplane $x_{0}^{\perp}$, and $C_{0}'$ is perimeter-minimizing in $K_{0}'$.
\end{lemma}
The proof of this lemma is omitted, as it can be obtained via a slight modification of \cite[Proposition 9.9]{Giu84book}.

\section{Proof of the main result}
We now dispose of all the necessary tools to prove Theorem \ref{thm:VS}. We argue by contradiction. By the results of the previous section, via a blow-up argument we can restrict the proof to the case of a domain $\Omega_{0}$ being a convex cone with vertex at the origin, and of a minimizer given by the intersection with $\Omega_{0}$ of a half-space whose boundary plane passes through the origin and meets $\de \Omega_{0}$ orthogonally. 

Before giving the proof of Theorem \ref{thm:VS}, we introduce a class of convex cones that will play a key role in the first part of the proof.
\begin{definition}
We say that a convex cone $C \subset \R^{3}$ is a pyramid provided there exist two wedges $W_{1}$, $W_{2}$ with orthogonally incident spines such that $C = W_{1} \cap W_{2}$.
\end{definition}
We note that a pyramid $C \subset \R^{3}$ is always a cone with vertex at the point $V_{0}$ in which the spines of the wedges intersect each other. Moreover, up to a rotation and a translation, there exist $a, \, b > 0$ such that
\begin{equation*}
    C = C_{a, b} := \{ x \in \R^{3} : \, x_{3} \geq \max\{ a |x_{1}|, b |x_{2}| \} \} \, .
\end{equation*}
\begin{proof}[Proof (of Theorem \ref{thm:VS}).] 
The proof is split into two steps. In the first, we show that a suitable plane through the vertex of a \emph{pyramid}, i.e., of a cone over a rectangle, cannot be area-minimizing in the pyramid itself. In the second, we employ a ``packing-box'' technique that allows us to reduce the case of a general convex cone to that of a suitably associated pyramid.

\textit{Step 1.}
Consider a pyramid cone $C_{a,b}$. We want to show that the plane 
	\[
	\pi_{0} = \{x\in \R^{3}:\ x_{1} = 0\} 
	\]
is not locally area-minimizing in $C_{a,b}$.
To do so, we build a family of competitors which improve the area of $\pi_{0}$ in $C_{a,b}$. For $\e \geq 0$, let $\pi_{\e} = \{x\in \R^{3}:\ x_{1} = \e\}$ and define
\begin{equation*}
R_{\e} := C_{a,b} \cap \{x\in \pi_{\e}:\ x_{3} \leq 1 \} \, , \qquad A_{\e} := \Hau^{2}(R_{\e}) \, .
\end{equation*}
We note that $R_{0}$ is a triangle in the plane $\pi_{0}$, and that $R_{\e}$ is a trapezium in the plane $\pi_{\e}$ whenever $0<\e<1$. Moreover, up to translations, $R_{\e}$ is obtained from $R_{0}$ by removing a triangle of area $\frac{a^{2}\e^{2}}{b}$, so that we have
\begin{equation} \label{eq:areaeps}
A_{\e} = A_{0} - \dfrac{a^{2}}{b} \e^{2} \, .
\end{equation}

The idea is now to connect the trapezium $R_{\e}$ with $\pi_{0} \cap C_{a,b}$ in order to obtain a local variation of $\pi_{0} \cap C_{a,b}$. We formulate the problem in the following way. Let $h>0$ and let $T_{h}$ be the trapezium defined as
    \begin{equation*}
        T_{h} = C_{a,b} \cap \{x\in \pi_{0}:\ 1 \leq x_{3} \leq 1 + h \} \, .
    \end{equation*}
    We immediately note that
    \begin{equation} \label{eq:areatrap}
        \Hau^{2}(T_{h}) = \dfrac{h (2+h)}{b} \, .
    \end{equation}
    We look for those smooth functions $\phi_{h}$ defined on the segment $\{ 1 \leq x_{3} \leq 1 + h \}$ satisfying the following conditions:
    \begin{equation} \label{eq:condphi}
        \phi_{h}(1) = 1 \, , \qquad \phi_{h}(1+h) = 0 \, .
    \end{equation}
    We observe that, looking at $\phi_{h}$ as a function of both variables $x_{2}$ and $x_{3}$ defined in $T_{h}$, the ruled surface
\begin{equation*}
G_{\e}(\phi_{h}) = \{ (\e \phi_{h}(x_{3}),x_{2},x_{3}): \, (x_{2},x_{3}) \in T_{h} \}
\end{equation*}
connects $R_{\e}$ with $\pi_{0} \cap C_{a,b}$. By suitably choosing $h$ and the map $\phi_{h}$, we claim that 
\begin{equation} \label{eq:condcompet}
A_{\e} + \Hau^{2}(G_{\e}(\phi_{h})) < A_{0} + \Hau^{2}(T_{h}) \, . 
\end{equation}
Using \eqref{eq:areaeps}, \eqref{eq:areatrap}, and the area formula, \eqref{eq:condcompet} turns out to be equivalent to
    \begin{equation} \label{eq:condcompet2}
        \iint_{T_{h}} \sqrt{1 + \e^{2} |\phi'_{h}(x_{3})|^{2}}\dd x_{2} \dd x_{3} < \dfrac{a^{2}}{b} \e^{2} + \dfrac{h (2 + h)}{b} \, .
    \end{equation}
In order to guarantee \eqref{eq:condcompet2}, we only need to impose that the second-order derivative at $0$ of the left-hand side is strictly smaller than the same derivative of the right-hand side. Differentiating both sides and applying Dominated Convergence, it suffices to choose $h$ and $\phi_h$ so that
    \begin{equation} \label{eq:condcompet4}
        \int_{1}^{1+h} t \, \phi'_{h}(t)^{2} \dd t < a^{2} \, .
    \end{equation}
    Then, for $\alpha > 0$, we choose
    \begin{equation*}
        \phi_{h} = \phi_{h,\alpha}(t) := \dfrac{(1 + h)^{\alpha}t^{-\alpha} - 1}{(1 + h)^{\alpha} - 1} \, .
    \end{equation*}
    We observe that $\phi_{h,\alpha}$ fulfills \eqref{eq:condphi}. Taking $\phi_{h} = \phi_{h,\alpha}$, condition \eqref{eq:condcompet4} becomes
    \begin{equation} \label{eq:condcompet5}
        \dfrac{\alpha}{2} \dfrac{(1+h)^{\alpha} + 1}{(1+h)^{\alpha} - 1} = \int_{1}^{1+h} t \, \phi'_{h,\alpha}(t)^{2} \dd t < a^{2} \, .
    \end{equation}
As $h\to +\infty$, the term on the left-hand side of \eqref{eq:condcompet5} tends to $\frac{\alpha}{2}$, hence it is enough to choose $\alpha < 2a^{2}$ and $h$ large enough to enforce \eqref{eq:condcompet5}. This ultimately proves \eqref{eq:condcompet} and shows that $\pi_{0}$ cannot be area-minimizing in $C_{a,b}$.

\textit{Step 2.} Let now $\Omega_{0}$ be a generic convex cone with vertex at the origin. Thanks to Theorem \ref{thm:geod}, and up to rotations, we may suppose that the boundary of the minimizer $E_{00}$ is the intersection of the plane $\pi_{0}$ with $\Omega_{0}$, hence there exists $b>0$ such that
\begin{equation*}
\de E_{00}\cap \Omega_{0}  = \{ (0,x_{2},x_{3}): \, x_{3} \geq b |x_{2}| \} \, .
\end{equation*}
Now, by Theorem \ref{thm:geod} we have
\[
\Omega_{0}\subset W_{1} := \{x\in \R^{3}:\ x_{3}\ge b|x_{2}|\}\,.
\] 
Since the origin is an isolated vertex for $\Omega_{0}$, it is not possible that $\de\Omega_{0}$ contains the whole line $\{(t,0,0):\ t \in \R\}$, hence there must exist $a>0$ such that the pyramid $C_{a,b}$ verifies either 
\begin{equation}\label{eq:mezzapiramide1}
\Omega_{0}\cap \{x_{1}\ge 0\}\ \subset\ C_{a,b}\cap \{x_{1}\ge 0\}
\end{equation}
or
\begin{equation}\label{eq:mezzapiramide2}
\Omega_{0}\cap \{x_{1}\le 0\}\ \subset\ C_{a,b}\cap \{x_{1}\le 0\}\,.
\end{equation}
We can assume for instance that \eqref{eq:mezzapiramide1} holds true, otherwise we simply flip the argument. We take $\e>0$ and set
\begin{equation*}
\hat{R}_{\e} := \Omega_{0} \cap \{x\in \pi_{\e}:\ x_{3} \leq 1 \} \, , \qquad \hat{A}_{\e} := \Hau^{2}(\hat{R}_{\e}) \, .
    \end{equation*}
    With the choice of suitable values $h$ and $\alpha$, we already know that the connection map $\phi_{h,\alpha}$ constructed in the proof of Theorem \ref{thm:geod} satisfies
\begin{equation*}
A_{\e} + \Hau^{2}(G_{\e}(\phi_{h,\alpha})) < A_{0} + \Hau^{2}(T_{h}) \end{equation*}
whenever $\e$ is small enough. Finally, we observe that
\begin{align*}
\hat{A}_{\e} + \Hau^{2}(G_{\e}(\phi_{h,\alpha}) \cap \Omega_{0}) &\leq A_{\e} + \Hau^{2}(G_{\e}(\phi_{h,\alpha})) \\
\hat{A}(0) &= A(0)\,,
\end{align*}
which shows that $\pi_{0}$ cannot be a minimizer in $\Omega_{0}$. This concludes the proof of Theorem \ref{thm:VS}.
\end{proof}

\end{document}